\documentclass[10pt]{amsart}

\usepackage{amssymb,amsmath,amsthm,amsfonts,enumerate}

\newcommand{\set}[2]{\{ #1 \mid #2 \}}
\newcommand{\assign}{\mathrel{:=}}
\newcommand{\logic}{\mathrm}
\newcommand{\IL}{\logic{IL}}
\newcommand{\ML}{\logic{ML}}

\newtheorem*{fact}{Fact}
\newtheorem*{lemma}{Lemma}

\newtheorem*{theorem}{Theorem}

\title[Levin's and Prucnal's theorems]{Levin's~and~Prucnal's theorems on Medvedev's~logic of finite problems}
\author{Adam P\v{r}enosil}
\address{Universitat de Barcelona, Departament de Filosofia}
\email{adam.prenosil@gmail.com}

\begin{document}

\begin{abstract}
  The purpose of this note is to provide a transparent and unified retelling of both Skvortsov's proof of the structural completeness of Medvedev's logic of finite problems, which is a classical result originally due to Prucnal, and of Levin's proof that Medvedev's logic of finite problems is the largest extension of the (weak) Kreisel--Putnam logic with the disjunction property. Presenting both results together allows us to simplify their presentation, as they both hinge on the same lemma. There is no novel content in this note, its purpose is merely to present the material in a more accessible way.
\end{abstract}

\maketitle

  The theorems of Levin~\cite{Levin} and Prucnal~\cite{Prucnal} are two of the most important results on Medvedev's (propositional super-intuitionistic) logic of finite problems. The only ambition of this note is to present their existing proofs in a more reader-friendly manner. In the case of Prucnal's theorem, we follow the proof of Skvortsov~\cite{Skvortsov}.\footnote{Levin's theorem as stated in~\cite{Levin} is that Medvedev's logic is the largest extension $\logic{L}$ of the Kreisel--Putnam logic with the property that if $\varphi \vee \psi$ is a theorem of $\logic{L}$, then either $\varphi$ or $\psi$ is a theorem of classical logic. The variant of Levin's theorem which we prove here is in fact due to Maksimova~\cite{Maksimova}, who however uses a different method of proof and does not cite~\cite{Levin}.}

  Let $M_{n}$ denote the free meet semilattice generated by the set $\{ 1, \dots, n \}$. That is, $M_{n}$ is a finite Boolean algebra with $n$ coatoms minus the top element, in which the set of coatoms of the Boolean algebra is identified with $\{ 1, \dots, n \}$. Each element of~$M_{n}$ has the form $\bigwedge I$ for some non-empty set $I \subseteq \{ 1, \dots n \}$, and $\bigwedge I \leq \bigwedge J$ in~$M_{n}$ if and only if $J \subseteq I$. The~family of posets $M_{n}$, viewed as Kripke frames for intuitionistic logic, defines Medvedev's logic of finite problems $\ML$.
  
  Medvedev's logic enjoys the \emph{disjunction property}: for all formulas $\varphi$ and $\psi$
\begin{align*}
  \vdash_{\ML} \varphi \vee \psi \implies \text{either } \vdash_{\ML} \varphi \text{ or } \vdash_{\ML} \psi.
\end{align*}
  This is because each pair of Medvedev's frames $M_{m}$ and $M_{n}$ is isomorphic to a pair of disjoint generated subframes of~$M_{m+n}$. If $\varphi$ fails in some valuation on $M_{m}$ and $\psi$ fails in some valuation on $M_{n}$, then combining these into a single valuation on $M_{m+n}$ yields a counter-example to $\varphi \vee \psi$.

  The \emph{weak Kreisel--Putnam logic} extends intuitionistic logic by the~axiom
\begin{align*}
  (\neg p \rightarrow (\neg q \vee \neg r)) \rightarrow ((\neg p \rightarrow \neg q) \vee (\neg p \rightarrow \neg r)).
\end{align*}
   This axiom is easily seen to be valid in Medvedev's logic. Indeed, Medvedev's logic validates the stronger \emph{Kreisel--Putnam axiom}
\begin{align*}
  (\neg p \rightarrow (q \vee r)) \rightarrow ((\neg p \rightarrow q) \vee (\neg p \rightarrow r)).
\end{align*}
  
  We now define the \emph{Kreisel--Putnam rank} of a formula (of intuitionistic logic). The~rank of a formula of the form $\neg \varphi$ is $1$. If $\varphi$ and $\psi$ are formulas of ranks $m$ and $n$ respectively, then $\varphi \vee \psi$ has rank $m + n$, $\varphi \wedge \psi$ has rank $m \cdot n$, and $\varphi \rightarrow \psi$ has rank $n^{m}$. All other formulas have a Kreisel--Putnam rank of $\infty$.

\begin{fact}
  If $\varphi$ is a formula of Kreisel--Putnam rank $n$, then there are $\psi_{1}$, \dots, $\psi_{n}$ such that $\varphi$ is equivalent to $\neg \psi_{1} \vee \dots \vee \neg \psi_{n}$ in the weak Kreisel--Putnam logic.
\end{fact}

\begin{proof}
  This is proved by an easy induction over the complexity of $\varphi$. The~weak Kreisel--Putnam axiom is used in the case of $\varphi \assign \psi \rightarrow \chi$: if~$\psi$ is equivalent to $\neg {\psi}_{1} \vee \dots \vee \neg {\psi}_{m}$ and $\chi$ is equivalent to $\neg {\chi}_{1} \vee \dots \vee \neg {\chi}_{n}$ in the weak Kreisel--Putnam logic, then $\psi \rightarrow \chi$ is equivalent to the conjunction of the $m$ formulas $\neg {\psi}_{i} \rightarrow \chi$. Each of these is equivalent to the disjunction of the $n$ formulas $\neg {\psi}_{i} \rightarrow \neg {\chi}_{j}$, each of which is in turn equivalent to the rank~$1$ formula $\neg (\neg {\psi}_{i} \rightarrow {\chi}_{j})$. By the distributive law, $\psi \rightarrow \chi$ is equivalent to a disjunction of $n^{m}$ of these rank~$1$ formulas.
\end{proof}

  For each $n$ we may choose intuitionistic formulas $\alpha^{n}_{1}, \dots, \alpha^{n}_{n}$ and a valuation $u_{n}$ on the Medvedev frame $M_{n}$ such that
\begin{enumerate}[(i)]
\item $\vdash_{\IL} \neg (\alpha^{n}_{i} \wedge \alpha^{n}_{j})$ for $i \neq j$,
\item $\vdash_{\IL} \neg \neg (\alpha^{n}_{1} \vee \dots \vee \alpha^{n}_{n})$, and
\item $i \in u_{n}(\alpha^{n}_{j}) \iff i = j$.
\end{enumerate}
  That is, the formulas $\alpha^{n}_{i}$ are pairwise inconsistent, exhaustive in a weak sense, and in the valuation $u_{n}$ each of them holds in exactly one of the maximal elements of the frame $M_{n}$ (namely, $\alpha^{n}_{i}$ holds in the maximal element $i$).

  For example, for $n = 2^{m}$ we may take the formulas $\alpha^{n}_{i}$ to be all con\-junctions of $m$ formulas which contain exactly one of the formulas $p_{j}$ or $\neg p_{j}$ for ${1 \leq j \leq m}$. For $n = 2^{m} - k$ we may simply combine the last $k+1$ of these into a disjunction. A suitable valuation $u_{n}$ is then easy to find.

  For a non-empty set $I \subseteq \{ 1, \dots, n \}$ we define $\alpha^{n}_{I} = \neg \neg \bigvee_{i \in I} \alpha^{n}_{i}$. Observe that
\begin{align*}
  \bigwedge I \in u_{n}(\alpha^{n}_{J}) \iff I \subseteq J.
\end{align*}

  We show that the valuation $u_{n}$ is universal in the sense that any valuation on $M_{n}$ can be expressed by composing $u_{n}$ with some substitution~$\sigma$, which can moreover be chosen so that its image consists of formulas of finite Kreisel--Putnam rank.

\begin{lemma}
  Let $v_{n}$ be a valuation on $M_{n}$. Then there is a substitution $\sigma$ such that $v_{n}(\varphi) = u_{n}(\sigma(\varphi))$ and $\sigma(\varphi)$ has finite Kreisel--Putnam rank for each formula $\varphi$.
\end{lemma}

\begin{proof}
  We define the substitution $\sigma$ on each atom $p$ as follows:
\begin{align*}
  \sigma(p) \assign \bigvee \set{\alpha^{n}_{I}}{\bigwedge I \in v_{n}(p)}.
\end{align*}
  Then $\bigwedge J \in u_{n} (\sigma (p))$ if and only if $\bigwedge J \in u_{n}(\alpha^{n}_{I}) \text{ for some } \bigwedge I \in v_{n}(p)$. But $\bigwedge J \in u_{n}(\alpha^{n}_{I})$ if and only if $\bigwedge I \subseteq \bigwedge J$. Thus $\bigwedge J \in u_{n}(\sigma(p))$ if and only if $\bigwedge J \in v_{n}(p)$, in other words $u_{n}(\sigma(p)) = v_{n}(p)$. Induction over the complexity of $\varphi$ immediately yields that $u_{n}(\sigma(\varphi)) = v_{n}(\sigma(\varphi))$. Since each formula of the form $\sigma(p)$ has finite Kreisel--Putnam rank, so does each formula of the form $\sigma(\varphi)$, by another induction over the complexity of $\varphi$.
\end{proof}

\begin{theorem}
  Medvedev's logic of finite problems is the largest axiomatic extension of the weak Kreisel--Putnam logic which has the disjunction property.
\end{theorem}

\begin{proof}
  Let $\logic{L}$ be such an extension of the weak Kreisel--Putnam logic. Suppose that $\nvdash_{\ML} \varphi(p_1, \dots, p_n)$. We show that $\nvdash_{\logic{L}} \varphi(p_1, \dots, p_n)$.

  By the above lemma, there is a substitution $\sigma$ such that $\sigma(p_{i})$ has finite Kreisel--Putnam rank for each $p_{i}$ and $\nvdash_{\ML} \varphi(\sigma(p_{1}), \dots, \sigma(p_{n}))$. By induction over the complexity of~$\varphi$, the formula $\varphi(\sigma(p_{1}), \dots, \sigma(p_{n}))$ has finite Kreisel--Putnam rank. That is, $\varphi(\sigma(p_{1}), \dots, \sigma(p_{n}))$ is equivalent to $\neg \psi_{1} \vee \dots \vee \neg \psi_{k}$ for some~$k$ in the weak Kreisel--Putnam logic, and therefore also in $\ML$ and in $\logic{L}$. Then $\nvdash_{\ML} \neg \psi_{1} \vee \dots \vee \neg \psi_{k}$, so $\nvdash_{\ML} \neg \psi_{i}$ for each $\psi_{i}$. All non-trivial super-intuitionistic logics agree on the provability of negated formulas, thus $\nvdash_{\logic{L}} \neg \psi_{i}$ for each $\psi_{i}$. The disjunction property of $\logic{L}$ yields that $\nvdash_{\logic{L}} \neg \psi_{1} \vee \dots \vee \neg \psi_{k}$, hence $\nvdash_{\logic{L}} \varphi(\sigma(p_{1}), \dots, \sigma(p_{n}))$. It follows that $\nvdash_{\logic{L}} \varphi(p_{1}, \dots, p_{n})$.
\end{proof}

  A super-intuitionistic logic $\logic{L}$ is \emph{structurally complete} if each admissible rule $\varphi \vdash \psi$ of $\logic{L}$ is valid in $\logic{L}$, where a rule $\varphi \vdash \psi$ is \emph{admissible} in $\logic{L}$ in case
\begin{align*}
  \vdash_{\logic{L}} \sigma(\varphi) \implies {\vdash_{\logic{L}} \sigma(\psi) \text{ for each substitution $\sigma$.}}
\end{align*}

\begin{theorem}
  Medvedev's logic of finite problems is structurally complete.
\end{theorem}

\begin{proof}
  Suppose that $\varphi \nvdash_{\ML} \psi$, as witnessed by a Kripke valuation $v_{n}$ on the Medvedev frame~$M_{n}$. Restricting to a generated subframe if necessary, we may assume that  $\varphi$ holds at each world of $M_{n}$ and $\psi$ fails at some world of $M_{n}$ in the valuation $v_{n}$. To prove the structural completeness of $\ML$, it will suffice to find a substitution~$\sigma$ such that $\vdash_{\ML} \sigma(\varphi)$ and $\nvdash_{\ML} \sigma(\psi)$. Let us take the substitution~$\sigma$ from the previous lemma. Then $\nvdash_{\ML} \sigma (\psi)$, as witnessed by the valuation $u_{n}$ on~$M_{n}$. It~remains to prove that $\vdash_{\ML} \sigma(\varphi)$.

  Suppose to the contrary that $\nvdash_{\ML} \sigma (\varphi)$, as witnessed by a valuation $w_{m}$ on the frame $M_{m}$. We then define the map $f\colon M_{m} \to M_{n}$ as follows:
\begin{align*}
  f (i) & \assign \text{the unique $j \in M_{n}$ such that } i \in w_{m}(\alpha_{j}), &
  f (\bigwedge I) & \assign \bigwedge f[I].
\end{align*}
  This map is well defined because each element of $M_{m}$ is a meet of a unique non-empty set of maximal elements. For each maximal element $j$ of $M_{m}$
\begin{align*}
  f(j) \in u_{n}(\alpha_{i}) & \iff j \in w_{m}(\alpha_{i}),
\end{align*}
  therefore for each non-empty $J \subseteq \{ 1, \dots, m \}$
\begin{align*}
  f(\bigwedge J) \in u_{n}(\alpha_{I}) & \iff \bigwedge I \leq f(\bigwedge J) = \bigwedge f[J] \\
  & \iff f[J] \subseteq I \\
  & \iff J \subseteq w_{m}(\alpha_{I}) \\
  & \iff \bigwedge J \in w_{m}(\alpha_{I}).
\end{align*}
  It follows that for each atom $p$
\begin{align*}
  f(\bigwedge J) \in u_{n}(\sigma(p)) \iff \bigwedge J \in w_{m}(\sigma(p)).
\end{align*}
 
  Now observe that the map $f$ is a \emph{p-morphism}: it is monotone, and if $f(\bigwedge I) = \bigwedge J \in M_{n}$ and $\bigwedge J \leq \bigwedge J'$ in $M_{n}$, then $\bigwedge I \leq \bigwedge I'$ and $f(\bigwedge I') = \bigwedge J'$ for some $\bigwedge I' \in M_{m}$. Namely, take $I' \assign I \cap f^{-1}[J']$. Then $\bigwedge I \leq \bigwedge I'$ and $f[I'] = J'$, since $J' \subseteq J$ and $f[I] = J$, and therefore $f(\bigwedge I') = \bigwedge J'$.
  
  Because $f$ is a p-morphism, it follows by induction over the complexity of~$\chi$ that
\begin{align*}
  f(\bigwedge J) \in u_{n}(\sigma(\chi)) \iff \bigwedge J \in w_{m}(\sigma(\chi))
\end{align*}
  for each formula $\chi$. Because $\sigma(\varphi)$ fails in some world of $M_{m}$ in the valuation $w_{m}$, also $\sigma(\varphi)$ must also fail in some world of $M_{n}$ in the valuation $u_{n}$. But this is a contradiction: $u_{n}(\sigma(\varphi)) = v_{n}(\varphi)$ and $\varphi$ holds in~$v_{n}$ at each world of $M_{n}$.
\end{proof}

\end{document}